\documentclass[10pt,psamsfonts]{amsart}
\usepackage{epsfig}
\usepackage{graphics}

\newcommand{\eref}[1]{(\ref{#1})}
\newtheorem{remark}{Remark}[section]
\newtheorem{lemma}{Lemma}[section]
\newtheorem{theorem}{Theorem}[section]
\newtheorem{proposition}{Proposition}[section]

\makeatletter

\@addtoreset{equation}{section}
\makeatother

\usepackage{epsfig}
\usepackage{graphics}
\usepackage{amsmath}
\usepackage{amssymb}
\usepackage{float}
\RequirePackage{amssymb}
\RequirePackage{latexsym}
\RequirePackage{subeqnarray}

\begin{document}
\thispagestyle{empty} \setcounter{page}{1}

\title[GnRH Generator Tuning]{Bifurcation-based parameter tuning in a model of the GnRH pulse and surge generator}

\author{Fr\'ed\'erique Cl\'ement \and Alexandre Vidal}

\thanks{This work is part of the large scale INRIA project REGATE (REgulation of the GonAdoTropE axis). \\
E-mail adresses : fredérique.clement@inria.fr, alexandre.vidal@inria.fr}

\maketitle

\vspace{-1cm}

\begin{center}
INRIA Paris-Rocquencourt Research Centre,

Domaine de Voluceau, Rocquencourt BP 105,

78153 Le Chesnay cedex, France
\end{center}

\begin{abstract}

We investigate a model of the GnRH pulse and surge generator, with the definite aim of constraining the model GnRH output with respect
to a physiologically relevant list of specifications. The alternating pulse and surge pattern of secretion results from the interaction
between a GnRH secreting system and a regulating system exhibiting fast-slow dynamics. The mechanisms underlying the behavior of the model
are reminded from the study of the Boundary-Layer System according to the "dissection method" principle.
Using singular perturbation theory, we describe the sequence of bifurcations undergone by the regulating (FitzHugh-Nagumo) system, encompassing
the rarely investigated case of homoclinic connexion. Basing on pure dynamical considerations, we restrict the space of parameter search for the
regulating system and describe a foliation of this restricted space, whose leaves define constant duration ratios between the surge and the pulsatility
phase in the whole system. We propose an algorithm to fix the parameter values to also meet the other prescribed ratios dealing with amplitude and
frequency features of the secretion signal. We finally apply these results to illustrate the dynamics
of GnRH secretion in the ovine species and the rhesus monkey.

\vspace{3mm}


\vspace{3mm}


\end{abstract}

\section{Introduction}

In vertebrates, the reproductive system is made up of the hypothalamus, belonging to the central nervous system, the pituitary gland and the gonads (ovaries in females, testes in males). Within the hypothalamus, specific neurons secrete the GnRH (Gonadotrophin Releasing Hormone) in a pulsatile manner. This pulsatility has a fundamental role in the differential control of the secretion of both gonadotropins by the pituitary gland: LH (luteinizing hormone) and FSH (follicle stimulating hormone). In females, the pulsatile pattern is tremendously altered once per ovarian cycle into a surge which triggers LH surge and ovulation in response to increasing levels of estradiol secreted by the ovaries. The estradiol signal is conveyed to GnRH neurons through a network of interneurons. The balance between stimulatory and inhibitory signals emanating from interneurons controls the behavior of the GnRH network. 

We have proposed a mathematical model accounting for the alternating pulse and surge pattern of GnRH secretion \cite{fc-jpf_07}. The model is based on the coupling between two systems running on different time scales. We thus consider the following four-dimensional dynamical system: 
\begin{subequations}
\begin{eqnarray}
\epsilon\delta\dot{x} &=& -y+f(x)  \label{eqdotx} \\
\epsilon\dot{y} &=& a_{0}x+a_{1}y+a_{2}+cX  \label{eqdoty} \\
\epsilon\gamma\dot{X} &=& -Y+g(X)  \label{eqdotX} \\
\dot{Y} &=& b_{0}X+b_{1}Y+b_{2}  \label{eqdotY}
\end{eqnarray}
\label{Sys1}
\end{subequations}
where $f$ and $g$ are two cubic functions :
\begin{eqnarray*}
f(x) &=& \nu _{0}x^{3}+\nu _{1}x^{2}+\nu _{2}x \\
g(X) &=& \mu _{0}X^{3}+\mu _{1}X^{2}+\mu _{2}X
\end{eqnarray*}

The faster system \eref{eqdotx}-\eref{eqdoty}, henceforth named ``GnRH Secreting System'', corresponds to the average activity of GnRH neurons,
while the slower one \eref{eqdotX}-\eref{eqdotY}, named ``Regulating System'', corresponds to the average activity of regulatory neurons.
The $x,X$ variables represent the neuron electrical activities (action potential), while the $y, Y$ variables relate to ionic and secretory dynamics.
In each system, the fast and slow variables feedback on each other. The coupling between both systems is mediated through the unilateral influence
of the slow regulatory interneurons onto the fast GnRH ones ($cX$ term in equation \eref{eqdoty}). The coupling term aggregates the global balance
between inhibitory and stimulatory neuronal inputs onto the GnRH neurons. The analysis of the slow/fast dynamics exhibited within and between both
systems allows to explain the different patterns (slow oscillations, fast oscillations and periodical surge) of GnRH secretion.

This model can be used as a basis to understand the control exerted by ovarian steroids (estradiol and progesterone) on GnRH secretion,
in terms of amplitude and frequency of oscillations and discriminate a direct (on the GnRH network) from an indirect action (on the regulatory network)
of steroid feedbacks. To account accurately for this control, we have to fully understand the sequence of bifurcations corresponding to the different
phases of GnRH secretion and the occurence of an hysteresis loop. Our main goal in this paper is thus to carry on with a deeper analysis of this model,
based on singular perturbation theory, to meet precise quantitative neuroendocrinological specifications. These specifications deal with
the duration of the luteal (progesterone-dominated) and follicular (estradiol-dominated) phases of the ovarian cycle, and the ratios between (i) the surge
duration and the whole cycle duration, (ii) the pulse amplitude and the surge amplitude, and (iii) the increase in pulse frequency from the luteal to the
follicular phase.

The paper is organized as follows. In section \ref{startingmodel}, we remind the main properties of the original model and explain
the mechanisms underlying the alternation between pulsatility phases and surges of the secretion signal.
In section \ref{RegSys}, we give a precise bifurcation diagram of the Regulating System according to the parameters
$\varepsilon $, $b_{1}$ and $b_{2}$, allowing to restrict the domain of parameter values.
We prove the existence and describe the shape of a foliation of the restricted domain,
whose leaves correspond to constant duration ratios.
This novel type of control form enables us to find the parameter values fulfilling any such prescribed ratio.
Section \ref{GnRHSecrSys} deals with the influence of other parameters on the remaining specifications.
It calls to methodological developments falling into the field of weakly coupled nonlinear oscillators.
Finally, in section \ref{Appli}, we enunce in details the quantitative specifications for the ovarian cycle
in the sheep and rhesus monkey, since direct measurements of GnRH are available in these species.
In each case, we instance the search for the relevant parameter values and illustrate the resulting
GnRH secretion signal from numerical simulations.

\section{Qualitative understanding of the original model}   \label{startingmodel}

In this section, we remind the main mechanisms underlying the model behavior, which is necessary to introduce the subsequent analysis in section 3. Following \cite{fc-jpf_07}, we assume that all parameter values are positive except $\nu _{0}$ and $\mu _{0}$.

\subsection{Reparameterization}

Some parameters of \eref{Sys1} are useless since, whatever
their values, we can fix the other parameter values to obtain the same
orbit. We can first remove the $b_{0}$ parameter through a simple rescaling,
which transforms system \eref{Sys1} into:
\begin{subequations}
\begin{eqnarray}
\varepsilon ^{\prime }\delta ^{\prime }\dot{x} &=& -y+f(x) \\
\varepsilon ^{\prime }\dot{y} &=& a_{0}x+a_{1}y+a_{2}+cX \\
\varepsilon ^{\prime }\gamma \dot{X} &=& -Y+g(X) \\
\dot{Y} &=& X+b_{1}^{\prime }Y+b_{2}^{\prime }
\end{eqnarray}
\end{subequations}
with:
\begin{eqnarray*}
\varepsilon ^{\prime }=b_{0}\varepsilon ,\qquad 
b_{1}^{\prime }=\frac{b_{1}}{b_{0}},\qquad b_{2}^{\prime }=\frac{b_{2}}{b_{0}}.
\end{eqnarray*}
We fix $b_{0}=1$ from now on. We can also remove the $\gamma $ parameter since
system \eref{Sys1} also reads:
\begin{subequations}
\begin{eqnarray}
\varepsilon ^{\prime \prime }\delta ^{\prime \prime }\dot{x} &=& -y+f(x) \\
\varepsilon ^{\prime \prime }\dot{y} &=& a_{0}^{\prime \prime }x+a_{1}^{\prime
\prime }y+a_{2}^{\prime \prime }+c^{\prime \prime }X \\
\varepsilon ^{\prime \prime }\dot{X} &=& -Y+g(X) \\
\dot{Y} &=& X+b_{1}Y+b_{2}
\end{eqnarray}
\end{subequations}
with:
\begin{eqnarray*}
\varepsilon ^{\prime \prime }=\varepsilon \gamma ,\qquad
\delta ^{\prime \prime }=\frac{\delta }{\gamma },\qquad 
a_{0}^{\prime \prime }=a_{0}\gamma,\qquad 
a_{1}^{\prime \prime }=a_{1}\gamma ,\qquad 
a_{2}^{\prime \prime}=a_{2}\gamma ,\qquad 
c^{\prime \prime}=c\gamma .
\end{eqnarray*}
We will fix $\gamma =1$ in the sequel.

\subsection{Qualitative behavior of the uncoupled system }

The average activity of neurons of both regulating and secreting neurons is modeled by a classical FitzHugh-Nagumo system,
while the control of GnRH neurons activity by regulatory neurons is expressed by means of a one-way coupling of variable $X$
onto variable $y$. Hence, we can analyze separately the slow-fast subsystem \eref{eqdotX}-\eref{eqdotY}, according to the time
scale parameter $\varepsilon $ and the $Y$-nullcline parameters $b_{0}$, $b_{1}$, and $b_{2}$.

The classical approach of slow-fast systems consists in studying the ``Boundary-Layer System'' obtained after
fixing the slow dynamics at $0$ (here $\dot{Y}=0$) and focusing on the geometric invariants of the fast dynamics.
Precisely, with the time rescaling:
\begin{equation}
t=\varepsilon \tau
\label{Rescal}
\end{equation}
subsystem \eref{eqdotX}-\eref{eqdotY} and $\left( RS_{\varepsilon } \right)$:
\begin{subequations}
\begin{eqnarray}
\dot{X} &=& -Y+g(X) \\
\dot{Y} &=& \varepsilon \left( X+b_{1}Y+b_{2} \right)
\end{eqnarray}
\label{dv}
\end{subequations}
are conjugate. Posing $\varepsilon =0$ in this last system, we obtain the corresponding Boundary-Layer System:
\begin{subequations}
\begin{eqnarray}
\dot{X} &=& -Y+g(X) \label{BLSeqdotX}\\
\dot{Y} &=& 0 \label{BLSeqdotY}
\end{eqnarray}
\label{BLSFN}
\end{subequations}
This dynamics is a zero-order approximation of system \eref{eqdotX}-\eref{eqdotY} according to $\varepsilon $.
The singular perturbation theory allows us to assess where this approximation is valid.

We consider the set of singular points of subsystem \eref{BLSFN}, defined by the cubic function $Y=g(X)$, and analyze their nature
according to the fast dynamics \eref{BLSeqdotX}. The two local extrema of the cubic are non hyperbolic points.
The points on the cubic middle branch (connecting the two local extrema) are hyperbolic repulsive, while the two other
(left and right) branches consist of hyperbolic attractive points. Each branch is a normally invariant hyperbolic manifold
of subsystem \eref{BLSFN}. They persist under perturbation by the slow dynamics (i.e. for small values of $\varepsilon $)
into a $O(\varepsilon )$-close normally invariant hyperbolic manifold of \eref{eqdotX}-\eref{eqdotY}. The attractive manifolds
attract the current point, as soon as it lies within their attraction basin, so that the orbit actually enters a $O(exp(-k/\varepsilon )$ neighborhood.

In the classical case we are interested in, the values of parameters $b_{1}$ and $b_{2}$ are small enough for the $Y$-nullcline to intersect the cubic
on the middle branch at a repulsive singular point of system \eref{eqdotX}-\eref{eqdotY} and the two other intersection points to be far away, respectively
on the left and right branch. It is well-known that, in such a case, subsystem \eref{BLSFN} admits an attractive limit cycle surrounding the middle
singular point (cf Figure \ref{PhasesAccX}-a).

\begin{figure}[htb]
\centering
\includegraphics[width=13cm]{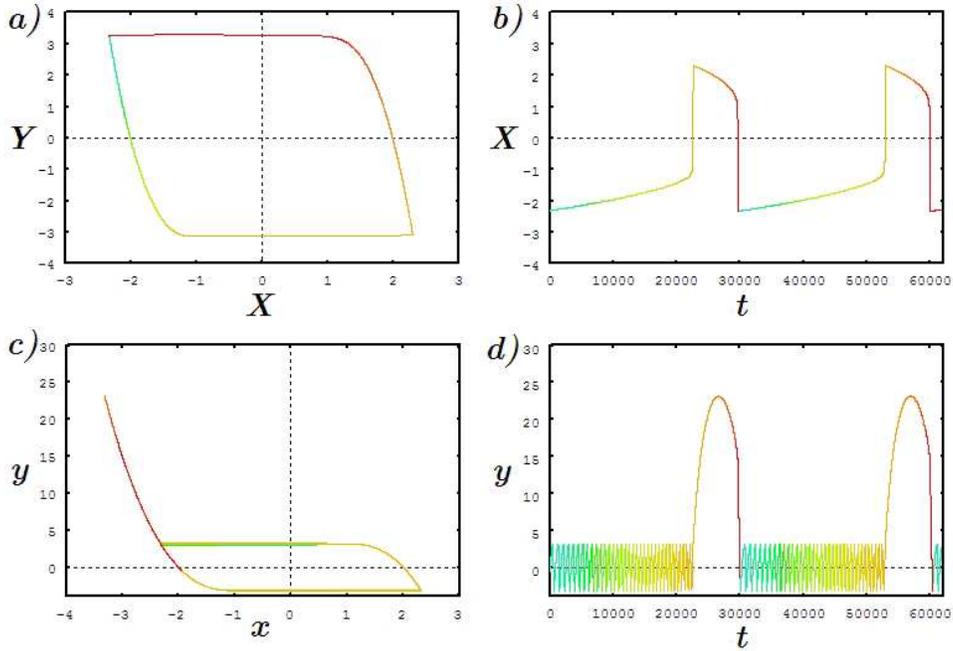}
\caption{Instance of orbit projections and signals generated by \eref{Sys1}. Each color corresponds to the same time mark in the four
panels for which the same parameterization has been used.
a) Attractive limit cycle of the Regulating System \eref{eqdotx}-\eref{eqdoty}. b) $X(t)$ output signal of the Regulating System.
c) Projection of system \eref{Sys1} orbit on the $(x,y)$-plane. d) $y(t)$ output signal of the GnRH Secreting system. \newline
The pulsatility phase duration (resp. surge duration) corresponds, within a
$O(\varepsilon )$-approximation, to the time during which $X<0$ (resp. $X>0$)
along the parameterization of the Regulating System limit cycle.}
\label{PhasesAccX}
\end{figure}

In the course of a period, the current point first comes near to, and then goes down slowly along the
left perturbed manifold (near the left branch of the cubic $Y=g(X)$), at a $O(1)$ speed. Near the extremum of the cubic (which is a non hyperbolic
singular point of \eref{BLSFN}), it keeps on going down and moves away from the cubic. At that time, it is subject to the fast motion and quickly
reaches (at a $O(1/\varepsilon )$ speed) the vicinity of the
right branch, as its $Y$-coordinate remains almost constant. Then it slowly goes up along the right perturbed manifold, before it finally comes back
quickly  near the left branch. This well-known motion is detailed in the proof of Lemma \ref{Secu}. For more details on singular perturbation theory,
we refer to \cite{fe_79} for the general study of slow-fast systems and \cite{ThVidal} for its application to oscillations.

\subsection{Qualitative behavior of the whole system}  \label{QualBehav}

We briefly remind here the specific behavior of system \eref{Sys1} and the corresponding pattern of
the $y(t)$ signal generated as a model output. We refer to \cite{fc-jpf_07} for more explanations about the model properties.

Through the coupling of \eref{eqdotx}-\eref{eqdoty} with \eref{eqdotX}-\eref{eqdotY}, the Regulating System drives the dynamics of the
GnRH Secreting System, by controling the position of its straight-line separatrix defined by:
\[
a_{0}x+a_{1}y+a_{2}+cX=0
\]
When all parameter values are fixed, the dynamics is determined by $X$ values. Let us consider $X$ as a parameter.
As its value increases, system \eref{eqdotx}-\eref{eqdoty} displays the following sequence of behaviors:
\begin{enumerate}
\item There is a unique singular point far up on the left branch.
\item An inverse saddle-node bifurcation occurs.
\item There are three singular points: a saddle far up on the left branch, an attractive point on the right branch and another saddle on its right. 
\item A supercritical Hopf bifurcation occurs.
\item One of the three singular points lies on the middle branch and is surrounded by an attractive relaxation limit cycle.
\item An inverse supercritical Hopf bifurcation occurs.
\item There are three singular points: a saddle far down on the right branch, an attractive point on the left branch and another saddle on its left.
Any orbit starting near the origin reaches the vicinity of the left branch and approaches the attractive singular point.
\item A saddle-node bifurcation occurs.
\item There is a unique singular point far down on the right branch. Any orbit starting near the origin reaches the vicinity of the left branch and goes up along.
\end{enumerate}

\begin{figure}[htb]
\centering
\includegraphics[width=14.5cm]{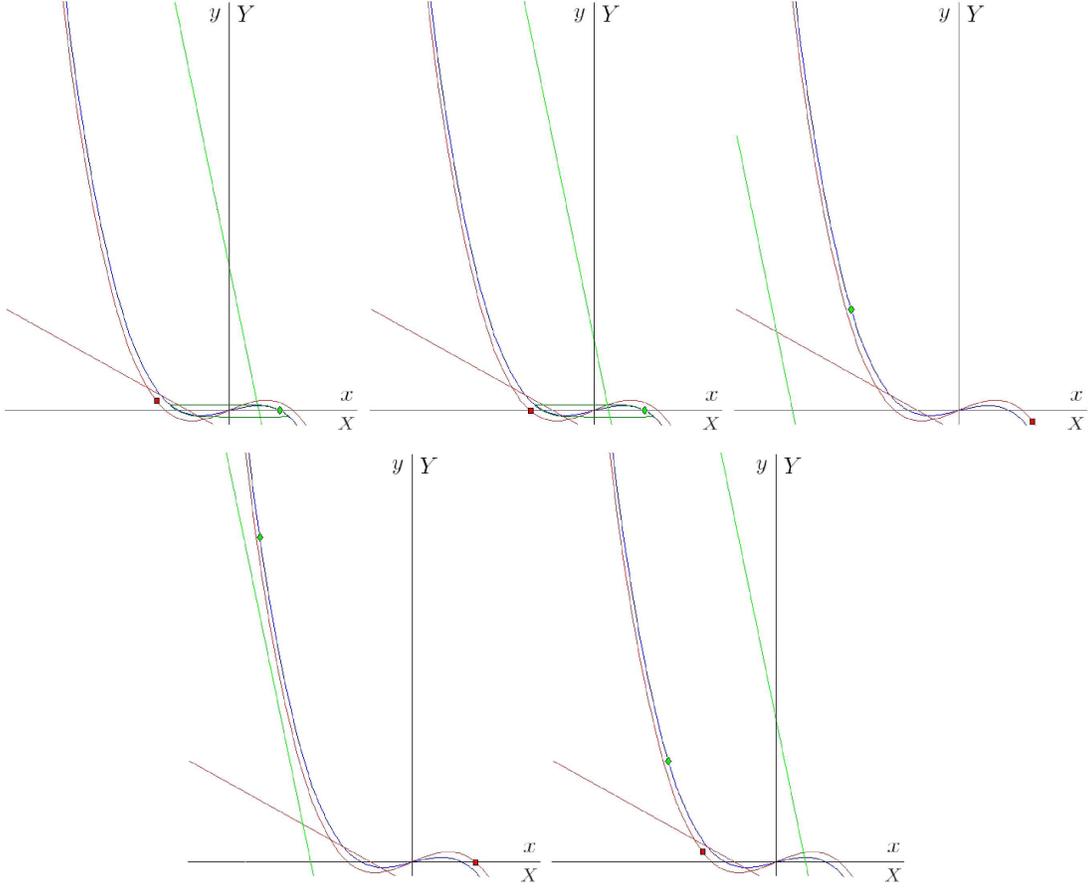}
\caption{GnRH Secreting System behavior as $(X,Y)$ describes the limit cycle of the Regulating System. We have set $Y=cst$ ($\varepsilon $ zero-order approximation).
The objects relating to the Regulating System are drawn in red or orange, those relating to the GnRH Secreting System are drawn in blue or green:
the $\dot{X}$-nullcline $Y=g(X)$ is the orange cubic, the $\dot{x}$-nullcline $y=f(x)$ is the blue cubic, the $\dot{Y}$-nullcline is the red line,
the $\dot{y}$-nullcline is the green line, the red square represents the $(X,Y)$ current point of the Regulating System, the green diamond represents
the $(x,y)$ current point of the GnRH Secreting System. \newline
a) Beginning of the pulsatility phase : the GnRH Secreting system displays the step 5) behavior of the sequence described in section \ref{QualBehav}.
b) Middle of the pulsatility phase.
c) Pulsatility to surge transition : the GnRH Secreting system displays the step 9) behavior.
d) Surge rise.
e) Surge decline : the GnRH Secreting system has come back to step 5).}
\label{Sketch}
\end{figure}

In the sequel, we will set precise conditions on the parameters such that, in the range of values taken by $X$ along the limit cycle of the
Regulating System, the GnRH Secreting System displays periodically the sub-sequence of behaviors described above from step 5 to 9 and back
from step 9 to 5. In this case, along an orbit of \eref{Sys1}, the current point $(x,y,X,Y)$ displays the following motion (represented in the panels of Figure \ref{Sketch}.

When $(X,Y)$ goes down near the left branch of the cubic $Y=g(X)$, $X$ increases slowly: system \eref{eqdotx}-\eref{eqdoty} displays the step-5 behavior
and $(x,y)$ turns around the limit cycle. Then, $X$ increases very quickly, the limit cycle of system \eref{eqdotx}-\eref{eqdoty} disappears (step 6).
$X$ decreases slowly and $(x,y)$ climbs up the left branch of the cubic $y=f(x)$ (step 7 to 9). Afterwards $X$ decreases very quickly and $(x,y)$
goes down near the left branch before oscillating again around the relaxation limit cycle. Typical
orbits of system \eref{Sys1} generate $y(t)$ signal patterns such as that represented in Figure \ref{PhasesAccX}-d.

\subsection{Pulse amplitude}  \label{PulseAmpl}

From the previous discussion, we have seen that a pulsatile $y(t)$ signal is generated from the oscillations of $(x,y)$ around the limit cycle
of the GnRH Secreting System. For all values of $X$ leading to such oscillations, the shape of the limit cycle of system \eref{eqdotx}-\eref{eqdoty}
remains almost the same (see Figure \ref{PhasesAccX}-c)). Consequently, the pulse amplitude is approximatively equal to the difference between
the $y$ coordinates of the two extrema $P^{f}_{+}$ and $P^{f}_{-}$ of $y=g(x)$. Thus, we can keep a unique relevant parameter to define
the cubic function $g$ and represent the amplitude. This reparameterization is described in more details in \cite{fc-jpf_07}.

To simplify the calculations and keep similar expressions for $f$ and $g$, we pose:
\begin{eqnarray*}
f(x) &=&-x^{3}+3\lambda ^{2}x \\
g(X) &=&-X^{3}+3\mu ^{2}X
\end{eqnarray*}
Then, the coordinates of the local extrema, $P_{\pm}^{f}$ and $P_{\pm}^{g}$, as well as those of their projections on the other branch of the cubics (along the trajectories of the fast systems), $Q_{\pm}^{f}$ and $Q_{\pm}^{g}$, are easy to compute:
\begin{subequations}
\begin{eqnarray}
P_{+}^{f} = (\lambda , 2\lambda ^{3}), \qquad P_{-}^{f} = (-\lambda , -2\lambda ^{3}) \\
P_{+}^{g} = (\mu , 2\mu ^{3}), \qquad P_{-}^{g} = (-\mu , -2\mu ^{3})
\end{eqnarray}
\label{Extr}
\end{subequations}
\begin{subequations}
\begin{eqnarray}
Q_{+}^{f} = (-2\lambda , 2\lambda ^{3}), \qquad Q_{-}^{f} = (2\lambda , -2\lambda ^{3}) \\
Q_{+}^{g} = (-2\mu , 2\mu ^{3}), \qquad Q_{-}^{g} = (2\mu , -2\mu ^{3})
\end{eqnarray}
\label{ProjExtr}
\end{subequations}
These reference points and their coordinates are displayed in Figure \ref{Cubics}.

\begin{figure}[htb]
\centering
\includegraphics[height=9cm]{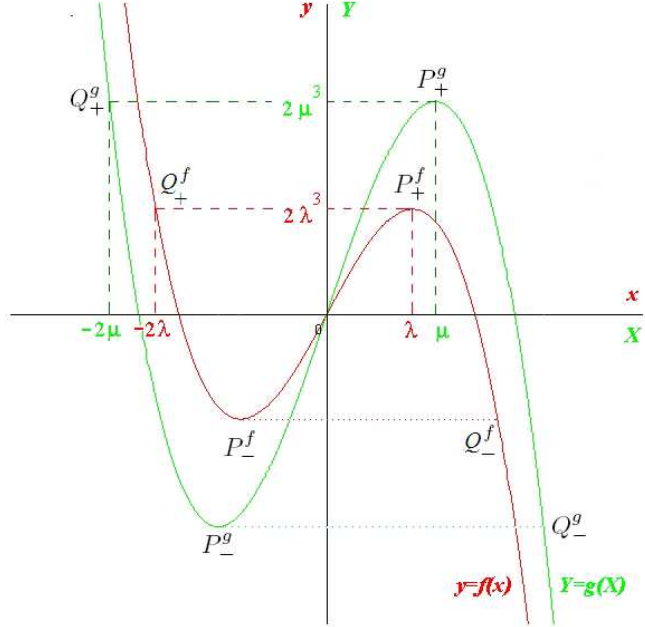}
\caption{Graph of the cubics defining the $\dot{x}$ and $\dot{X}$ nullclines. The coordinates of
the local extrema $P^{f}_{\pm }$ (resp. $P^{g}_{\pm }$) of the $x$-nullcline (resp. $X$-nullcline) are given in
\eref{Extr}. The projections $Q^{f}_{\pm }$ (resp. $Q^{g}_{\pm }$) of these extrema along $y=cst$ (resp. $Y=cst$)
are given in \eref{ProjExtr}.}
\label{Cubics}
\end{figure}

Henceforth, we set $\lambda$ and $\mu$ to given, constant values. We also fix $\delta$ to a value small enough for the GnRH Secreting System to admit a relaxation limit cycle in the step-5 case described in section \ref{QualBehav}. Each of the other parameters constitute a degree of freedom, whose value has to be fitted to any specific list of specifications, as it is illustrated in \S \ref{Appli}. We begin in the next section by studying the effect of $\varepsilon $, $b_{1}$ and $b_{2}$ on the duration ratio between the surge and the whole cycle.

\subsection{Control of the two-way transition between the pulsatility phase and the surge by the dynamics of the regulating system} \label{PhasesChar}

From the sequence of behaviors described in \S\ref{QualBehav}, we can establish a direct link between
the pattern of the GnRH secretion signal and the values of $X$ along two distinct parts of the
Regulating System limit cycle.

As stated in \cite{fc-jpf_07}, the time during which the current point $(X,Y)$ remains in a neighborhood
of the right branch of $Y=g(X)$ along the limit cycle is $O(1)$ in $\varepsilon $. In contrast, the transition
time of the Regulating System from a neighborhood of $P_{+}^{g}$ to a neighborhood of $Q_{+}^{g}$ is $O(\varepsilon )$.
The same $\varepsilon $-order applies respectively to the slow motion near the left branch of the cubic $Y=g(X)$
and the fast motion between $P_{-}^{g}$ and $Q_{-}^{g}$. 

In a precise zone of the parameter space defined in \S\ref{GnRHSecrSys}, we can ensure that,
for $-2\mu \leq X \leq -\mu$, the GnRH secretion signal has a pulsatile pattern (step 5), while, for $\mu \leq X \leq 2\mu$,
it switches to a surge pattern (step 7 to 9). Once $(X,Y)$ has reached the vicinity of $Q_{+}^{g}$,
it takes a $O(\varepsilon )$-time to $(x,y)$ to reach a neighborhood of $Q_{-}^{f}$ and oscillate again around
the limit cycle of the GnRH Secreting System. Consequently, if we neglect the $O(\varepsilon )$-duration of
fast motions, we can approximate the duration of the pulsatility phase by the time during which $X<0$
and the surge duration by the time during which $X>0$ (see panels c and d of Figure \ref{PhasesAccX}).

\section{Bifurcation Diagram of the Regulating System}  \label{RegSys}

In this section, we first describe the sequence of bifurcations undergone by the
Regulating System $\left( RS_{\varepsilon } \right)$, defined by \eref{dv} 
and the corresponding behaviors exhibited by the whole system (whether such behaviors
are meaningful or not from the biological viewpoint) in delimited zones of the parameter space.
Then, we restrict the parameter space to a domain where the whole system meets the prescribed dynamics,
alternating between pulsatility phases and surges.

We are interested in determining conditions under which the Regulating System admits an attractive limit cycle,
further denoted by $C(b_{1},b_{2}, \varepsilon )$. We have first to assume that system \eref{dv} admits a singular point
$\left( X_{0}, Y_{0}=g\left( X_{0}\right) \right) $ such that $\left\vert X_{0}\right\vert <\mu $, and that
$\varepsilon $  is small enough (in a sense that we will explain later in Lemma \ref{Secu}).
Yet, even under those necessary conditions, two different types of bifurcation may make the limit cycle disappear: a
supercritical Hopf bifurcation and a homoclinic bifurcation. To explain the occurrence of these
bifurcations, we have to remind that, for any $b_{1}>0$, system \eref{dv} admits two other singular points of saddle type
$\left( X_{-},Y_{-}=g(X_{-})\right) $, $X_{-}<-\mu $ and $\left(X_{+},Y_{+}=g(X_{+})\right) $, $X_{+}>\mu $.

\subsection{Hopf Bifurcation}

The Hopf bifurcation happens when the middle singular point $\left( X_{0},Y_{0}\right) $
coincides with the local minimum of the cubic $Y=g(X)$,
i.e. when $X_{0}=-\mu $. Then :
\[
-\mu +b_{1}g(-\mu )+b_{2}=0\Longleftrightarrow -\mu +b_{1}\left( \mu
^{3}-3\mu ^{3}\right) +b_{2}=0
\]
Hence, the Hopf bifurcation occurs along the surface:
\begin{equation}
\mathcal{H}_{p}:b_{2}=h_{p}(b_{1})=\mu +2b_{1}\mu ^{3}  \label{dv Hopf}
\end{equation}
in the parameter space $\left( \varepsilon , b_{1}, b_{2} \right) $.

The Hopf bifurcation gives birth to a small limit cycle which surrounds the middle
singular point $\left( X_{0},Y_{0}\right) $ as $(b_{1},b_{2})$ crosses transversally
$\mathcal{H}_{p}$. Since ($RS_{\varepsilon }$) is a slow-fast system, this limit cycle becomes quickly a big relaxation limit cycle when
$\left( \varepsilon, b_{1}, b_{2} \right)$ moves away from $\mathcal{H}_{p}$.

The occurrence of small limit cycles in such relaxation systems
is known as ``Canard phenomenon''. The ``Canard cycles'' are attractive limit cycles which follow,
against all expectations, a neighborhood of the fast dynamics manifold of repulsive points for a while.
Some examples are displayed in Figure \ref{Canards}.

\begin{figure}[htb]
\centering
\includegraphics[height=8cm]{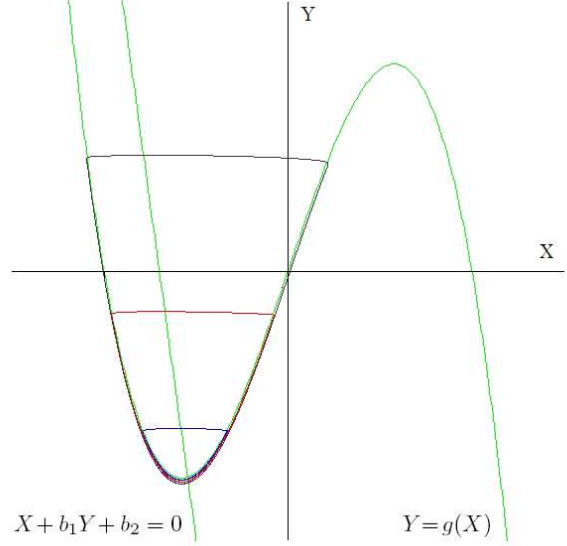}
\caption{Examples of Canard cycles of the Regulating System $(RS_{\varepsilon })$ for three different yet very close (to $10^{-12}$) values of $\varepsilon$.}
\label{Canards}
\end{figure}

\subsection{Homoclinic Bifurcation}

The homoclinic bifurcation occurs when the left saddle point $\left( X_{-},Y_{-}\right) $
\textquotedblleft crosses\textquotedblright\ the limit
cycle, leading to the intersection of its stable manifold with its unstable manifold. We
can intuitively understand this bifurcation by considering the common
limit periodic set $\Gamma _{0}$ approached by the family of periodic orbits
for the Hausdorff distance\footnote{
The Hausdorff distance between two compacts $K$ and $K^{\prime }$ of a metric
space $(E,d)$ is the smallest $r>0$ such that $\forall x\in K,d(x,K^{\prime })<r$
and $\forall y\in K^{\prime },d(y,K)<r$.}, when $\varepsilon \rightarrow 0$.
This set is the compact union of the cubic parts from $Q_{-}^{g}$ to $P_{+}^{g}$ and from
$Q_{+}^{g}$ to $P_{-}^{g}$ on one hand, and the fast trajectories $\left] P_{+}^{g} , Q_{+}^{g} \right[ $
and $\left] P_{-}^{g} , Q_{-}^{g} \right[ $ on the other hand;
it is represented in Figure \ref{lps}.

\begin{figure}[htb]
\centering
\includegraphics[height=5.5cm]{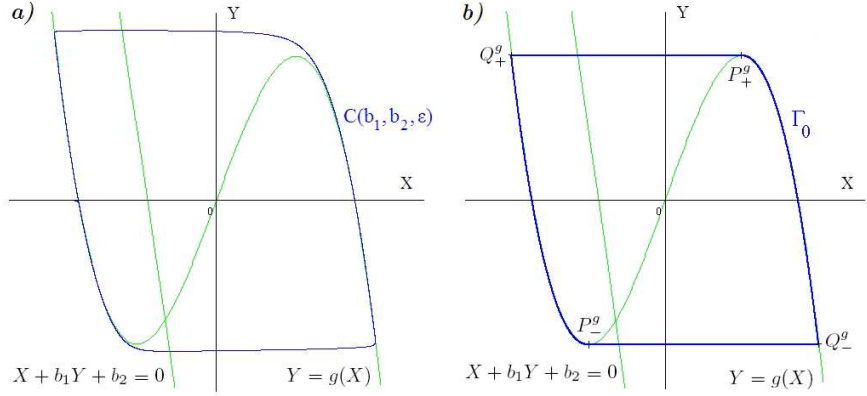}
\caption{a) Limit cycle of the Regulating System $(RS_{\varepsilon })$ for the following values of parameters: $b_{1}=0.1$, 
$b_{2}=1$, $\protect\varepsilon $ $=0.1$. b) Limit Periodic Set for 
$b_{1}=0.1 $, $b_{2}=1$: limit of the family $(C(b_{1},b_{2},\protect%
\varepsilon ))_{\protect\varepsilon \in ]0;\protect\varepsilon _{0}]}$ as $%
\protect\varepsilon $ tends to $0$ according to the Hausdorff distance.}
\label{lps}
\end{figure}

Hence the $\varepsilon$, zero-order approximation of the limit cycle allows
us to develop the zero-order approximation of the homoclinic connexion:
\[
\left( X_{-},g(X_{-})\right) =(X_{Q_{+}},g(X_{Q_{+}}))
\]
It is equivalent to state that $Q_{+}=\left( X_{Q_{+}},g(X_{Q_{+}})\right)
=(-2\mu ,2\mu ^{3})$ lies on the separatrix $X+b_{1}Y+b_{2}=0$:
\[
-2\mu +b_{1}g(-2\mu )+b_{2}=0
\]
so that the homoclinic bifurcation occurs along the surface:
\begin{equation}
\mathcal{H}_{c}^{0}:b_{2}=h_{c}^{0}(b_{1})=2\mu -2b_{1}\mu ^{3}
\label{Homoc0}
\end{equation}

To delimit the domain in the parameter space where the $\varepsilon$ zero-order approximation is valid,
we have to take into account the Canard phenomenon occurring near the Hopf bifurcation. In particular,
the occurrence of small cycles such as those represented in Figure
\ref{Canards} shows that, in the region of the parameter space close to the Hopf bifurcation surface
$\mathcal{H}_{p}$, $\mathcal{H}_{c}^{0}$ does not approximate the homoclinic bifurcation condition.
The occurrence of homoclinic bifurcations in such a case is an
awkward problem which we will not tackle in this article. Consequently, we make
the suitable assumptions to remain in a zone of parameter values for which this
question is not bothering. To do so, we need the following lemma:

\begin{lemma} \label{lemmaH0c}
For each $\alpha >0$, for all $(b_{1},b_{2})$ verifying:
\begin{equation}
b_{2}<\mu +2b_{1}\mu ^{3}-\alpha  \label{SecuHopf}
\end{equation}
and:
\begin{equation}
b_{2}<2\mu -2b_{1}\mu ^{3}  \label{SecuHomoc}
\end{equation}
there exists $\varepsilon _{0}>0$ such that, for all $\varepsilon \in
]0,\varepsilon _{0}[$, the limit cycle of \eref{dv} exists and contains
some points of $[ \mu ,+\infty [ \times \mathbb{R} $.
Moreover, for $(b_{1},b_{2})$ fixed, the limit cycle $C(b_{1},b_{2},
\varepsilon )$ lies in a $O(\varepsilon ^{2/3})$-neighborhood of $\Gamma_{0} $.
\label{Secu}
\end{lemma}

\begin{proof}
We use the sketch of the well-known proof of limit cycle existence for the
FitzHugh-Nagumo system, when the left and right singular points are respectively far away
from $P^{g}_{-}$ and $P^{g}_{+}$. This proof is based on the series expansion of the first
return map induced by the flow of  $(RS_{\varepsilon })$.
We focus here on the well-definition of this map. We base on the arguments enunciated in \cite{bo}
to derive the series expansion of the map and prove the existence of one stable fixed point and that
of the attractive limit cycle. The geometric objects mentioned below are represented in Figure \ref{NormHypMan}.

\begin{figure}[htb]
\centering
\includegraphics[height=9cm]{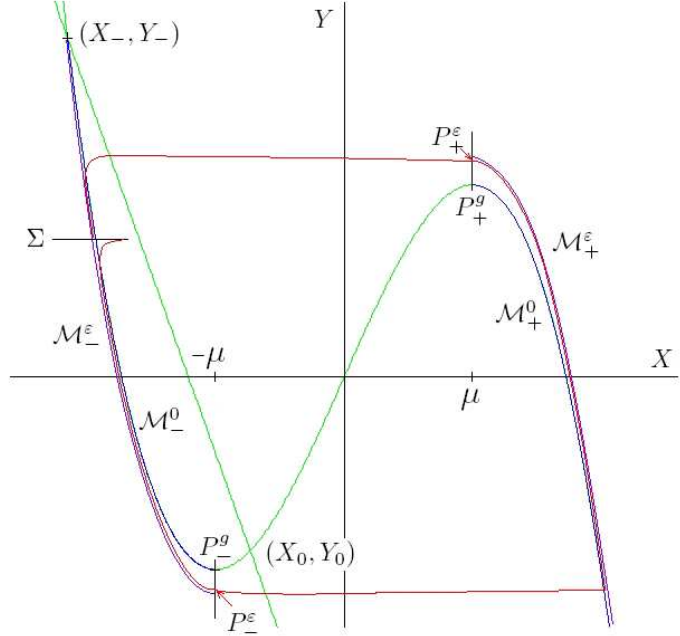}
\caption{The part $\mathcal{M}_{-}^{0}$ of the cubic $Y=g(X)$ between $\left( X_{-},Y_{-} \right) $
and $P^{g}_{-}$ is a normally invariant attractive manifold for the system $(RS_{0})$.
The part of the cubic $Y=g(X)$ between $P^{g}_{+}$ and $\left(X_{+},Y_{+}\right) $
(not shown on the picture, far down on the right branch) is a
normally invariant attractive manifold for the system $(RS_{0})$.
Thus, $\mathcal{M}_{-}^{0}$ (resp. $\mathcal{M}_{+}^{0}$) perturbs into a $O(\varepsilon )$-close
manifold $\mathcal{M}_{-}^{\varepsilon }$ (resp. $\mathcal{M}_{+}^{\varepsilon }$), which is invariant
and normally hyperbolic for  $(RS_{\varepsilon })$. \newline
Starting from an initial data on section $\Sigma $ transverse to $\mathcal{M}_{-}^{0}$,
we can track the orbit of $(RS_{\varepsilon })$ until $X=-\mu $, then near $\mathcal{M}_{+}^{\varepsilon }$
until $X=\mu $, and, at least, to $\Sigma $.}
\label{NormHypMan}
\end{figure}

Let $\alpha >0$ and $(b_{1},b_{2})$ verifying \eref{SecuHopf} and 
\eref{SecuHomoc}. The part of the cubic between  $\left(X_{-},Y_{-}\right) $ and $P^{g}_{-}$,
denoted by $\mathcal{M}_{-}^{0}$, is a normally attractive manifold of the
boundary-layer system associated to $(RS_{\varepsilon }$), which perturbs into a left
$O(\varepsilon )$-close normally attractive manifold of $(RS_{\varepsilon })$, denoted by
$\mathcal{M}_{-}^{\varepsilon }$. Let us consider a small section $\Sigma $ transverse
to $\mathcal{M}_{-}^{0}$. Then, all trajectories of $(RS_{\varepsilon })$
starting from $\Sigma $ must remain in an exponentially small neighborhood
of $\mathcal{M}_{-}^{\varepsilon }$ until they enter a $O(\varepsilon
^{2/3})$-neighborhood of $P^{g}_{-}$ (see \cite{ThVidal}). According to \eref{SecuHopf},
the middle singular point of $\left( RS_{\varepsilon } \right) $ is above $Q^{g}_{+}$.
Thus, we can choose $\varepsilon $ small enough so that this neighborhood does not
contain $\left(X_{0},Y_{0}\right) $.

Let us remind too that the limit cycle crosses $X=-\mu $ at a point $P_{-}^{\varepsilon }$
below the point $P^{g}_{-}$, since $\dot{y}<0$ along this part of the limit cycle.
All trajectories of $(RS_{\varepsilon })$ starting from a $O(\varepsilon
^{2/3})$-neighborhood of $P^{g}_{-}$ keep away from any singular point for 
$\varepsilon $ small enough. Hence, they remain in a $O(\varepsilon )$
-neighborhood of the fast trajectory starting from $P_{-}^{\varepsilon }$,
which reaches the right attractive manifold.

With the same arguments as above, we keep on tracking the orbits of 
$(RS_{\varepsilon })$ in an exponentially small neighborhood of the right $O(\varepsilon )$-close
normally attractive manifold of $(RS_{\varepsilon })$, denoted by $\mathcal{M}_{+}^{\varepsilon }$,
until they reach a point $P_{+}^{\varepsilon }$ lying on $X=\mu $ and above $P^{g}_{+}$. 
According to the theory of regular perturbations,
the orbits which have just followed the perturbed attractive manifold remain in a 
$O(\varepsilon ^{2/3})$-neighborhood of the fast trajectory connecting
$P^{g}_{+}$ with $Q^{g}_{+}$.
From \eref{SecuHomoc}, the left saddle point is
strictly above the point $Q_{+}$ on the cubic. This allows us to choose 
$\varepsilon $ small enough to make sure that all trajectories cross the
cubic and thereafter remain in an exponentially small neighborhood of 
$\mathcal{M}_{-}^{\varepsilon }$ until it crosses $\Sigma $. Thus the first return
map is well-defined on $\Sigma $ for $\varepsilon $ small enough,
i.e. $\varepsilon \in ]0,\varepsilon _{0}[$. The limit cycle 
$C(b_{1},b_{2},\varepsilon )$ lies in a $O(\varepsilon ^{2/3})$-neighborhood
of the limit periodic set $\Gamma _{0}$ (see Figure \ref{lps}) and contains
some points of $[\mu ,+\infty [ \times \mathbb{R} $.
\end{proof}

\begin{remark}  \label{RemR0}
The smaller $\alpha $ is, the smaller $\varepsilon _{0}$ has to be, to
avoid any Canard effect. Hence,  from now on, $\alpha $ is fixed to a constant value small enough so that
the region:
\[
\mathcal{R}_{0}=\left\{ \left( b_{1},b_{2},\varepsilon \right) \left\vert 
\begin{array}{l}
b_{2}<\mu +2b_{1}\mu ^{3}-\alpha \\ 
b_{2}<2\mu -2b_{1}\mu ^{3}
\end{array}
\right. \right\}
\]
is non empty. 

On the other hand, the closer $b_{2}$ is from $2\mu-2b_{1}\mu ^{3}$,
the closer the left singular point is from the limit
cycle, and the smaller $\varepsilon _{0}$ has to be. $\varepsilon _{0}$
tends to $0$ as $(b_{1},b_{2})$ approaches $\mathcal{H}_{c}^{0}$.
\end{remark}

In other words, for a fixed value of $\varepsilon $, a homoclinic connexion of the left
singular point occurs as $(b_{1},b_{2})$ approaches $\mathcal{H}_{c}^{0}$. We seek this
connexion as $b_{2}$ increases from $0$, assuming that:
\[
b_{1}>\frac{\mu +\alpha }{4\mu ^{3}}
\]
When $\varepsilon $ is fixed to a small value, there exists $(b_{1},b_{2})$,
verifying \eref{SecuHopf} and \eref{SecuHomoc}, for which the value of 
$\varepsilon _{0}$, defined in lemma \ref{lemmaH0c}, is greater than 
$\varepsilon $. Then, for $b_{2}$ small enough, the left singular point lies
outside the $O(\varepsilon ^{2/3})$-neighborhood of $\Gamma _{0}$ containing the limit cycle
$C(b_{1},b_{2},\varepsilon )$. For a value of 
$b_{2}$ close to (and smaller than) $2\mu -2b_{1}\mu ^{3}$, the saddle
undergoes a homoclinic connexion where the limit cycle disappear.

The following proposition enunciates the condition for a homoclinic bifurcation to occur
away from any Canard situation :

\begin{proposition}
There exists, locally near $\varepsilon =0$, a $C^{1}$-surface in $\left(
b_{1},b_{2},\varepsilon \right) $ of homoclinic connexions given by
the graph:
\[
\mathcal{H}_{c}:b_{2}=h_{c}\left( b_{1},\varepsilon \right) 
=2\mu -2\mu ^{3}b_{1}+O\left(
\varepsilon ^{2/3}\right)
\]
defined for:
\[
b_{1}\in \left] \frac{\mu +\alpha }{4\mu ^{3}},\frac{1}{\mu ^{2}}\right[
\]
\end{proposition}

\begin{proof}
We have seen that the homoclinic connexion occurs when the
left singular point, $\left( X_{-},Y_{-}\right) $, with $X_{-}<-\mu 
$, crosses the limit cycle. Let $ X_{\min }\left(b_{1},b_{2},\varepsilon \right)$ be the
minimum of $X$ along the limit cycle $C(b_{1},b_{2}, \varepsilon )$.
From the fast dynamics (which is null on the cubic), we have immediately that
$\left(X_{\min }(b_{1},b_{2},\varepsilon ), g(X_{\min }(b_{1},b_{2},\varepsilon )\right)$ lies on
the cubic. Consequently, the condition for the homoclinic connexion reads:
\[
X_{\min }(b_{1},b_{2},\varepsilon )+b_{1}g(X_{\min }(b_{1},b_{2},\varepsilon
))+b_{2}=0
\]
As we know that $X_{\min }(b_{1},b_{2},\varepsilon )=X_{Q_{+}}+O\left(
\varepsilon ^{2/3}\right) =-2\mu +O\left( \varepsilon ^{2/3}\right) $, it also
reads:
\begin{equation}
u(b_{1},b_{2},\varepsilon )=-2\mu +O(\varepsilon ^{2/3})+b_{1}g(-2\mu
+O(\varepsilon ^{2/3}))+b_{2}=0 \\
\label{Homoceps1}
\end{equation}
where $u$ is a $C^{1}$ function on its definition domain (where the limit
cycle exists). For each $b_{1}$ value such that: 
\begin{equation}
b_{1}\in \left] \frac{\mu +\alpha }{4\mu ^{3}},\frac{1}{\mu ^{2}}\right[
\label{Secub1}
\end{equation}
$X_{\min }(b_{1},2\mu -2b_{1}\mu ^{3},0)$ is well defined (it is equal to
$X_{Q_{+}}$) and:
\[
\frac{\partial u}{\partial b_{2}}(b_{1},2\mu -2b_{1}\mu ^{3},0)=1
\]
Then, the implicit function theorem implies the existence of a unique root
of $u$ with respect to $b_{2}$ for any $b_{1}$ verifying \eref{Secub1} and 
$\varepsilon $ small enough. From \eref{Homoceps1}, this root can be expanded as:
\[
b_{2}=2\mu -2\mu ^{3}b_{1}+O\left(
\varepsilon ^{2/3}\right),\qquad \varepsilon \rightarrow 0
\]
\end{proof}

\subsection{Parameter space reduction}

From the bifurcation diagram, we can reduce the parameter space in order to
impose a periodic behaviour to the Regulating System. We can also ensure that,
along its attractive periodic orbit, the variable $X$ alternatively takes negative and positive values.
Since the sign of $X$ along the orbit of \eref{dv} discriminates between 
the two phases (pulsatility phase and surge) of the whole system \eref{Sys1}, the restriction
of the value domain for $\left(\varepsilon, b_{1}, b_{2}\right)$ is the first step to obtain an alternation
of pulsatility phases and surges.

In the sequel, we will only consider $(\varepsilon ,b_{1},b_{2})$ to be under both the surface of homoclinic
connexion and a security plane on the underside of  the Hopf bifurcation:
\[
H_{p}^{-\alpha }:b_{2}=h_{p}(b_{1})-\alpha 
\]
Hence we assume that:
\[
(b_{1},b_{2},\varepsilon )\in \mathcal{R}_{1}=\left\{ \left(
b_{1},b_{2},\varepsilon \right) \left\vert 
\begin{array}{l}
b_{2}<h_{p}(b_{1})-\alpha \\ 
b_{2}<h_{c}(b_{1},\varepsilon ) \\ 
\varepsilon <\varepsilon _{0}
\end{array}
\right. \right\}
\]
where $\varepsilon _{0}$ is defined in Lemma \ref{Secu}.

\begin{figure}[htb]
\centering
\includegraphics[height=7cm]{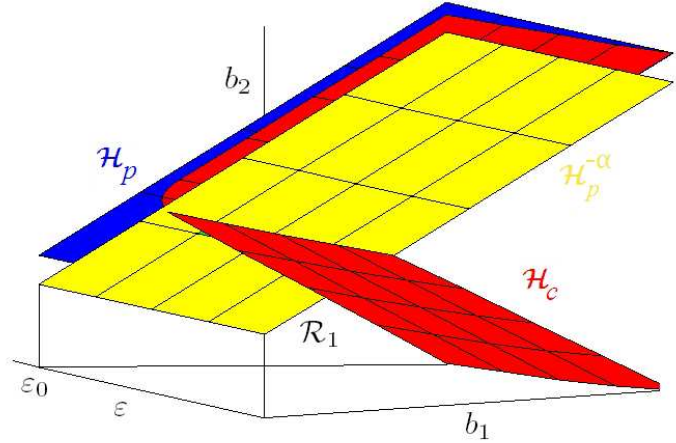}
\caption{Bifurcation diagram for $\varepsilon <\varepsilon _{0}$. We restrict the
parameter space to $\mathcal{R}_{1}$ : under the surface of homoclinic bifurcations $\mathcal{H}_{c}$
and far from the surface of Hopf bifurcation $\mathcal{H}_{p}$. There, we are ensured that the
Regulating System
exhibits an attractive limit cycle and that, along this orbit, $X$ takes alternatively positive
and negative values. Both conditions are necessary to obtain the alternance of surges and pulsatility phases
for the $y(t)$ signal generated by the whole system \eref{Sys1}.}
\label{Bif}
\end{figure}

In the next section, we will prove the possibility to tune the parameters $b_{1}$, $b_{2}$ and
$\varepsilon $ in this domain to fulfill the specification prescribing a given ratio of the duration of the pulsatility phase
to the surge duration.

\section{Foliation of the Regulating System parameter space} \label{FolRS}

We focus in this section on the way of determining all the values of the triplet
$(b_{1},b_{2},\varepsilon )$ leading to a given ratio between the duration of the pulsatility
phase and the surge duration. As we have seen in \S \ref{PhasesChar},
these durations are related to the time during which the current point $(X,Y)$ lies either in
the half-plane $X>0$  (pulsatility phase) or $X<0$ (surge) within one period along the limit
cycle of $(RS_{\varepsilon})$. Their ratio remains unchanged after time rescaling.

\begin{lemma}
For $(b_{1},b_{2},\varepsilon )\in \mathcal{R}_{1}$, let us consider a
parameterization $(X(t),Y(t))$ of the limit cycle $C(b_{1},b_{2},\varepsilon
)$ such that $X(0)=0$ and $Y(0)>0$, and let $T(b_{1},b_{2},\varepsilon )$
be its period. Then there exists a unique $T_{-}(b_{1},b_{2},\varepsilon )$
such that:
\begin{eqnarray*}
\forall t &\in &[0,T_{-}(b_{1},b_{2},\varepsilon )],\qquad X(t)\leq 0 \\
\forall t &\in &[T_{-}(b_{1},b_{2},\varepsilon ),T(b_{1},b_{2},\varepsilon)],\qquad X(t)\geq 0
\end{eqnarray*}
Let us denote $T_{+}(b_{1},b_{2},\varepsilon )=T(b_{1},b_{2},\varepsilon
)-T_{-}(b_{1},b_{2},\varepsilon )$. Then, as $\varepsilon $ tends to 0:
\begin{eqnarray*}
&&T_{-}(b_{1},b_{2},\varepsilon )=
\int_{X_{\min }(b_{1},b_{2},\varepsilon )}^{-\mu }\frac{g^{\prime}(X)+\frac{\partial \Psi _{-}}{\partial X}(X,\varepsilon )}
{X+b_{1}\left(g(X)+\Psi _{-}(X,\varepsilon )\right) +b_{2}}dX+O(\varepsilon ) \\
&&T_{+}(b_{1},b_{2},\varepsilon )=
\int_{X_{\max }(b_{1},b_{2},\varepsilon )}^{\mu }\frac{g^{\prime}(X)+\frac{\partial \Psi _{+}}{\partial X}(X,\varepsilon )}
{X+b_{1}\left(g(X)+\Psi _{+}(X,\varepsilon )\right) +b_{2}}dX+O(\varepsilon )
\end{eqnarray*}
where $\Psi _{-}<0$ and $\Psi _{+}>0$ are differentiable functions defined
respectively on:
\[
]X_{\min }(b_{1},b_{2},\varepsilon ),-\mu ]\times]0,\varepsilon _{0}]
\]
and:
\[
[X_{\min }(b_{1},b_{2},\varepsilon ),-\mu ]\times]0,\varepsilon _{0}]
\]
such that:
\begin{equation}
\exists \lambda (\varepsilon )\underset{\varepsilon \rightarrow 0}{=}
O(\varepsilon ^{2/3}),\forall X\in \lbrack X_{\min }(b_{1},b_{2},\varepsilon
),-\mu ],\left\vert \Psi _{\pm }(X,\varepsilon )\right\vert <\lambda
(\varepsilon )
\label{MajPsi}
\end{equation}
\end{lemma}

\begin{proof}
We consider only the left part of the limit cycle, along which $X<0$, and prove
the existence of $\Psi _{-}$. The argument for the right part is identical.

Let us assume $(b_{1},b_{2},\varepsilon )\in \mathcal{R}_{1}$ and consider
the set:
\[
U=\left\{ (X,Y)|X<-\mu ,Y<g(X_{\min }(b_{1},b_{2},\varepsilon
),Y<g(X)\right\}
\]
Since $\dot{X}=0$ along the left branch of the critical manifold and $\dot{Y}<0$
in $\bar{U}$, any orbit of \eref{dv} starting in $U$ escapes from $U$
across the half-line $\left\{ (X,Y)|X=-\mu ,Y<0\right\} $. 
From the dynamics \eref{dv}, we have also, along any trajectory remaining in $U$:
\begin{equation}\label{YdiffX}
\frac{dY}{dX}=\frac{\dot{Y}}{\dot{X}}=\varepsilon \frac{X+b_{1}Y+b_{2}}{-Y+g(X)}<0
\end{equation}
It follows from \ref{YdiffX} that such a trajectory can be represented as the graph of a
differentiable function of $X$. In particular, the limit cycle $
C(b_{1},b_{2},\varepsilon )$ goes through the point:
\[
\left( X_{\min}(b_{1},b_{2},\varepsilon ),g(X_{\min }(b_{1},b_{2},\varepsilon ))\right),
\]
enters $U$ and remains in $U$ until it crosses the half-line $\left\{
(X,Y)|X=-\mu ,Y<0\right\} $. Let us represent this trajectory as:
\[
\tilde{C}_{-}(b_{1},b_{2},\varepsilon ):Y=\gamma _{b_{1},b_{2},\varepsilon
}(X),\qquad X\in ]X_{\min }(b_{1},b_{2},\varepsilon ),-\mu ]
\]
and pose:
\[
\Psi _{-}:(X,\varepsilon )\rightarrow \gamma _{b_{1},b_{2},\varepsilon
}(X)-g(X)
\]

We can directly deduce that $\Psi _{-}<0$ and is differentiable with respect to $X$.
Since $(b_{1},b_{2},\varepsilon )$ lies in $\mathcal{R}_{1}$, the one-parameter
family $(C(b_{1},b_{2},\varepsilon ))_{\varepsilon \in ]0,\varepsilon _{0}]}$
converges to the limit periodic set $\Gamma _{0}$ described in Lemma \ref{Secu} and
each $\tilde{C}_{-}(b_{1},b_{2},\varepsilon )$ trajectory, $\varepsilon \in
]0,\varepsilon _{0}]$, remains in a $O\left( \varepsilon ^{2/3}\right)$-neighborhood of
$\{(X,g(X))|X\in ]X_{\min }(b_{1},b_{2},\varepsilon ),-\mu
]\}$. Finally, $\Psi _{-}$ is differentiable on:
\[
]X_{\min}(b_{1},b_{2},\varepsilon ),-\mu ]\times ]0,\varepsilon _{0}]
\]
and verifies
\eref{MajPsi}.
\end{proof}

\noindent It is worth noticing that, as $\varepsilon $ tends to $0$, the durations 
$T_{-}(b_{1},b_{2},\varepsilon )$ and $T_{+}(b_{1},b_{2},\varepsilon )$ tend
to the limit durations:
\begin{eqnarray}
&&T_{-}^{0}(b_{1},b_{2})=
\int_{-2\mu }^{-\mu }\frac{g^{\prime }(X)}{X+b_{1}g(X)+b_{2}}dX
\label{T0moins} \\
&&T_{+}^{0}(b_{1},b_{2})=
\int_{2\mu }^{\mu }\frac{g^{\prime }(X)}{X+b_{1}g(X)+b_{2}}dX
\label{T0plus}
\end{eqnarray}
which were chosen as approximations to compute the period of the limit cycle in \cite{fc-jpf_07}. 
When $(b_{1},b_{2},\varepsilon )$ is far under the surface of homoclinic
bifurcations $\mathcal{H}_{c}$, it is indeed a good approximation
(with a controlled $O(\varepsilon ^{2/3})$-error). But, in our case, since the duration of the
pulsatility phase has to be much longer than the surge duration, we have to increase
$T_{-}(b_{1},b_{2},\varepsilon )$ in comparison with $T_{+}(b_{1},b_{2},\varepsilon )$.
To do so, the current point has to be confined for a while in the vicinity of the left
singular point, where $(b_{1},b_{2},\varepsilon )$ may be very close to $\mathcal{H}_{c}$.
This is why we cannot neglect the time spent along the path from $X=X_{\min }$ to $X=-2\mu$,
even more so since the motion is very slow near the singular point. 

We now explain how to select the value of $b_{2}$ from fixed values of
$b_{1}$ and $\varepsilon $  to meet any prescribed ratio $T_{-}/T_{+}=r$.
We first consider the case $\varepsilon =0$.

\begin{proposition}
There exists a $C^{1}$-foliation of $\mathcal{R}_{1}\cap \{\varepsilon =0\}$
of one dimensionnal leaves such that:\newline
1) each leaf is the graph:
\[
b_{2}=l_{r}^{0}(b_{1}),\qquad b_{1}\in \left[ \bar{b}_{1}^{r},\frac{1}{\mu
^{2}}\right[ 
\]
of a differentiable function $l_{r}^{0}$,\newline
2) for each $r\geq 1$, there is a leaf on which:
\begin{equation}
\frac{T_{-}^{0}(b_{1},b_{2})}{T_{+}^{0}(b_{1},b_{2})}=r  \label{Ratio0}
\end{equation}
\end{proposition}

\begin{proof}
From \eref{T0moins} and \eref{T0plus}, we can see that $T_{-}^{0}(b_{1},b_{2})$
and $T_{+}^{0}(b_{1},b_{2})$ are continuous functions in 
$\mathcal{R}_{1}\cap \{\varepsilon =0\}$ and differentiable with respect to
$(b_{1},b_{2})$. Let us notice that for any value of $b_{1}$:
\[
\frac{T_{-}^{0}(b_{1},0)}{T_{+}^{0}(b_{1},0)}=1
\]
This comes from the central symmetry of the dynamics with respect to the
origin when $b_{2}=0$.

Besides, for any fixed value of $b_{1}$, $T_{-}^{0}(b_{1},b_{2})$ increases
and $T_{+}^{0}(b_{1},b_{2})$ decreases as $b_{2}$ increases.
Thus:
\begin{equation}
\frac{\partial T_{-}^{0}(b_{1},b_{2})}{\partial b_{2}}T_{+}^{0}(b_{1},b_{2})-
\frac{\partial T_{+}^{0}(b_{1},b_{2})}{\partial b_{2}}T_{-}^{0}(b_{1},b_{2})>0
\label{DiffAccb2}
\end{equation}

Let us first consider $b_{1}$ fixed such that:
\[
\frac{\alpha +\mu }{4\mu ^{3}} \leq b_{1} < \frac{1}{\mu ^{2}}
\]
We can increase $b_{2}$ until $(b_{1},b_{2})$ reaches the line $\mathcal{H}_{c}$.
When $b_{2}$ tends to $2\mu -2\mu ^{3}b_{1}$, 
$T_{-}^{0}(b_{1},b_{2})$ tends to $+\infty $ whereas $T_{+}^{0}(b_{1},b_{2})$
remains finite. Together with \ref{DiffAccb2},
this proves that, for each $r\geq 1$, there exists a unique solution 
$b_{2}=l_{r}^{0}(b_{1})$ in $\mathcal{R}_{1}\cap \{\varepsilon =0\}$ of \eref{Ratio0}.
As $T_{\pm }^{0}(b_{1},b_{2})$ are $C^{1}$-functions with respect to $(b_{1}, b_{2})$, the function
$l_{r}^{0}$ is $C^{1}$ with respect to $b_{1}$.

On the other hand, let us consider the time ratio when $(b_{1},b_{2})$ lies
on the security plane $H_{p}^{-\alpha }$, i.e.:
\[
b_{1}<\frac{\alpha +\mu }{4\mu ^{3}},\qquad b_{2}=h_{p}(b_{1})-\alpha
\]
Then:
\begin{eqnarray*}
T_{-}^{0}(b_{1},h_{p}(b_{1})-\alpha )=
\int_{-2\mu }^{-\mu }\frac{g^{\prime }(X)}{X+b_{1}g(X)+\mu+2b_{1}\mu ^{3}-\alpha }dX \\
T_{+}^{0}(b_{1},h_{p}(b_{1})-\alpha )=
\int_{2\mu }^{\mu }\frac{g^{\prime }(X)}{X+b_{1}g(X)+\mu +2b_{1}\mu^{3}-\alpha }dX
\end{eqnarray*}
Let us define $\bar{b}_{1}^{r}$ as the unique solution in $b_{1}$ of the equation:
\[
\frac{T_{-}^{0}(b_{1},h_{p}(b_{1})-\alpha )}{T_{+}^{0}(b_{1},h_{p}(b_{1})-\alpha )}=r
\]
Then, if we fix $r>1$, for each $b_{1}$ such as
\[
\bar{b}_{1}^{r} \leq b_{1} \leq \frac{\alpha +\mu }{4\mu ^{3}}
\]
there exists a unique solution $b_{2}=l_{r}^{0}(b_{1})$ of \ref{Ratio0}.
\end{proof}

\begin{theorem}
There exists a $C^{1}$-foliation of $\mathcal{R}_{1}$, defined near $\varepsilon =0$, of two dimensionnal
leaves such that:\newline
1) each leaf is the graph:
\[
b_{2}=l_{r}(b_{1},\varepsilon ),\qquad b_{1}\in \left[ \bar{b}_{1}^{r},
\frac{1}{\mu ^{2}}\right[ 
\]
of a differentiable function $l_{r}$,\newline
2) for each $r\geq 1$, there is a leaf on which:
\[
\frac{T_{-}(b_{1},b_{2},\varepsilon )}{T_{+}(b_{1},b_{2},\varepsilon )}=r
\]
\end{theorem}

\begin{proof}
We prolong $T_{\pm }(b_{1},b_{2},\varepsilon )$
into a differentiable function at $\varepsilon =0$ by posing $T_{\pm }(b_{1},b_{2},0 )=T_{\pm }^{0}(b_{1},b_{2})$.
Let us fix $r>1$ and $\bar{b}_{1}$ such that $\bar{b}_{1}^{r} < \bar{b}_{1} < \frac{1}{\mu ^{2}}$.
Then, like in the previous proof:
\[
\frac{\partial T_{-}(\bar{b}_{1},l_{r}^{0}(b_{1}), 0)}{\partial b_{2}}T_{+}(\bar{b}_{1},l_{r}^{0}(b_{1}), 0)-
\frac{\partial T_{+}(\bar{b}_{1},l_{r}^{0}(b_{1}), 0)}{\partial b_{2}}T_{-}(\bar{b}_{1},l_{r}^{0}(b_{1}), 0)>0
\]
The implicit function theorem implies that there exists an open subset $U^{r}_{\bar{b}_{1}}$ in the $(\varepsilon ,b_{1})$-space,
$\varepsilon \geq 0$, which contains $(0,\bar{b}_{1})$ and a differentiable function $l_{r}$ such that, for each $(\varepsilon , b_{1})$ in
$U^{r}_{\bar{b}_{1}}$ :
\[
\frac{T_{-}(b_{1},b_{2},\varepsilon )}{T_{+}(b_{1},b_{2},\varepsilon )}=r
\]
admits a unique solution defined by:
\[
b_{2}=l_{r}(b_{1},\varepsilon )
\]
By recovering all values of $\bar{b}_{1}$ such that $\bar{b}_{1}^{r} < \bar{b}_{1} < \frac{1}{\mu ^{2}}$, we define
$l_{r}$ for $b_{1}$ in $\left[ \bar{b}_{1}^{r}, \frac{1}{\mu ^{2}}\right[$ and $\varepsilon_{0} $ small enough.
\end{proof}

For any prescribed ratio $r$ between the duration of the pulsatility phase and the surge duration,
we have thus shown that there exists a 2-dimensional manifold of solutions in the parameter space
$(b_{1},b_{2},\varepsilon )$, and we have provided a $O(\varepsilon ^{2/3})$ approximation of the leaf by the plane:
\[
\left\{(\varepsilon, b_{1}, b_{2})|0 \leq \varepsilon \leq \varepsilon _{0}, b_{2}=l_{r}^{0}(b_{1}), \bar{b}_{1}^{r} \leq b_{1} \leq \frac{1}{\mu
^{2}} \right\}
\]

\begin{figure}[htb]
\centering
\includegraphics[height=7cm]{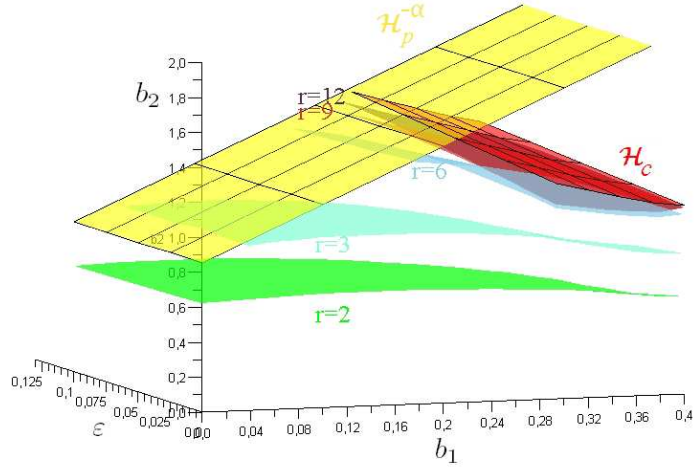}
\caption{Leaves in $(\varepsilon , b_{1}, b_{2})$-space defined by $T_{-}/T_{+}=r$, for $r=2,3,6,9,12$.
This ratio corresponds to the pulsatility phase duration over the surge duration.}
\label{Foliation}
\end{figure}

We can find numerically the relevant surface solutions for different ratios, as it is illustrated in
Figure \ref{Foliation} for $r=2,3,6,9$ and $12$. The highest $r$ is, the closest the leaf is to the surface of homoclinic connexions.

\section{Tuning of the parameters of the GnRH Secreting System}   \label{GnRHSecrSys}

In this section, we describe the process for tuning the values of  $\varepsilon $, $a_{0}$, $a_{1}$, $a_{2}$, and $c$ with respect to a definite list of quantitative specifications. We first state the necessary conditions to ensure that system \eref{Sys1} displays the alternation between pulsatility and surge described in \S \ref{QualBehav}. Using geometrical tools, we then analyze the effect of each parameter on each prescribed ratio (duration of the surge over the whole cycle duration, pulse amplitude over surge amplitude, and increase in pulse frequency from the luteal to the follicular phase). Finally, we describe the algorithmic procedure to apply in order to tune the parameter values one after another and  fulfill the specifications.

\subsection{Constraints on the parameters to keep the right order in the sequence of secretion patterns}\label{ordertuning}

We first study the parameter conditions under which the $y(t)$ signal does start to pulse right after the surge. We have to prevent the GnRH Secreting System from exploring the step 1-to-4 behaviors described in \S \ref{QualBehav}, wherever $(X,Y)$ is along the limit cycle of the Regulating System. In other words, the Hopf bifurcation (step 4) of the GnRH Secretion System has to occur for a value of $X$ such that $X \leq -2\mu $. This condition amounts to say that for $X=-2\mu$, $\dot{y}$ is positive at the extremum $P^{f}_{+}$, so that we have to impose the following inequality on $c$:
\begin{equation}
c \leq \frac{a_{0}\lambda +2a_{1}\lambda ^{3}+a_{2}}{2\mu }  \label{NonRightHopf}
\end{equation}

We also have to ensure that the pulsatility phase does not break
off before the Regulating System switches to the fast motion from $P_{-}^{g}$
to $Q_{-}^{g}$ that induces the surge. The GnRH Secreting System has to keep on oscillating long enough, so that the step 6 reverse Hopf bifurcation cannot occur for values of $X$ strictly smaller than $-\mu$. This condition reads:
\begin{equation}
c \geq \frac{-a_{0}\lambda -2a_{1}\lambda ^{3}+a_{2}}{\mu }  \label{NonLeftHopf1}
\end{equation}

The GnRH Secreting System has to remain in the surge mode (where the $\dot{y}$-nullcline lies on the right of $P_{-}^{f}$) until $X$ undergoes the fast motion between $P_{+}^{g}$ and $Q_{+}^{g}$ that brings the secreting system back to the pulsatile mode. This again imposes a condition on the reverse Hopf bifurcation that cannot occur for a value of $X$ strictly greater than $\mu $:
\begin{equation}
c \geq \frac{a_{0}\lambda +2a_{1}\lambda ^{3}-a_{2}}{\mu }  \label{NonLeftHopf2}
\end{equation}
Bringing together the two last conditions, we obtain the following constraint:
\begin{equation}
c \geq \left\vert \frac{a_{0}\lambda +2a_{1}\lambda ^{3}-a_{2}}{\mu } \right\vert  \label{Constraintc}
\end{equation}

Finally, we have to take care that the current point $(x,y)$ does cross the $\dot{y}$-nullcline from the left to the right when $(X,Y)$ undergoes the fast motion from $P_{+}^{g}$ to $Q_{+}^{g}$.
To this end, the slope of the $\dot{y}$-nullcline ($a_{0}/a_{1}$) has to be steep enough, hence the value of $a_1$ has to be small enough. Otherwise the current point $(x,y)$ might escape from control at the end of the surge (step 9 behavior described in \S \ref{QualBehav}) and go on climbing up the left branch of the cubic $y=f(x)$, on the left side of the $\dot{y}$-nullcline.

\subsection{Effect of the $\varepsilon$, $a_0$ and $a_2$ on the ratios prescribed to the secretion signal}   \label{Influ}
\paragraph{Effect of $\varepsilon$}
When $a_{1}$ is small enough, we can approximate the period of the GnRH Secreting System,
for any value of $X$ ranging between $-2\mu$ and $-\mu$, by:
\begin{equation}
\varepsilon \left( \int_{-2\lambda }^{-\lambda }\frac{f^{\prime }(x)}{%
a_{0}x+a_{1}f(x)+a_{2}+cX}dx+\int_{2\lambda }^{\lambda }\frac{f^{\prime }(x)}{%
a_{0}x+a_{1}f(x)+a_{2}+cX}dx\right)
\label{FreqPulseAccX}
\end{equation}
Under conditions \eref{NonRightHopf} and \eref{Constraintc}, this period is both well-defined and finite. We can deduce from expression \eref{FreqPulseAccX} that the pulse frequency increases as $\varepsilon$ decreases. Similarly, the surge amplitude increases as $\varepsilon$ decreases, since the $(x,y)$ point moves along the left branch of the cubic $y=f(x)$ at a $O(1/\varepsilon)$ speed, compared to the $O(1)$-speed  of the $Y$ motion, that we use as reference.

\paragraph{Effect of $a_0$ and $a_2$}The surge amplitude increases exponentially as $a_{0}$ decreases. Thanks to this exponential dependency, it is easy to find an $a_{0}$ value compatible with  condition \eref{NonRightHopf} and leading to a high-amplitude surge. $a_{2}$  controls the position of the stripe described by $\dot{y}$ in the $(x,y)$-plane, as $X$ increases from $-2\mu$ to $-\mu$.  When  $a_{2}$ increases, the stripe moves leftwards in the $(x,y)$-plane, so that both the pulse frequency and the surge amplitude increase.
\paragraph{Effect of $c$}
 $c$ controls the width of this stripe. Since conditions \eref{NonRightHopf} and \eref{Constraintc} reduce drastically the range of admissible values for $c$, they also lower its influence on the $y$ dynamics along the left branch of the cubic $y=f(x)$. Hence, changes in $c$ do not impact much on the surge. In contrast, a small change in the width of the stripe impacts a lot on the set of limit cycles followed by $(x,y)$ as $X$ increases from $-2\mu$ to $-\mu$.
Finally, the greater $c$ is (even subject to condition \eref{NonRightHopf}), the greater are the stripe width and the pulse frequency ratio between the beginning and the end of the pulsatility phase.

\subsection{Procedure for tuning the parameter values} \label{Algo}
As it is explained in \S \ref{PulseAmpl}, we have fixed once and for all the values of $\delta $, $\lambda $ and $\mu $. More precisely, $\delta=0.0125$ and the values of $\lambda$ and $\mu$ correspond to cubic functions $f$ and $g$ expressed as:
\begin{subequations}
\begin{eqnarray}
f(x)=-x^{3}+2.5x\\
g(X)=-X^{3}+4X
\end{eqnarray}
\end{subequations}

We propose here an algorithm-like procedure for tuning the other parameters in order to meet the set of quantitative specifications together. The procedure consists of the following steps:
\begin{enumerate}
\item  Fix the value of $a_{2}$, since this parameter has multiple influences on the different ratios. 
To ensure that conditions \eref{NonRightHopf} and \eref{Constraintc} will be fulfilled, we have to position the $\dot{y}$-nullcline precisely in the $(x,y)$ plan in the uncoupled case ($X=0$), so that it  separates the $P^{f}_{-}$ extremum, on its left, from the intersection point between the $x$-axis and the cubic $y=f(x)$, on its right. The corresponding $a_{2}$ values range between $\lambda $ and $\sqrt(3) \lambda $. Within this interval, $a_2$ will be even greater that the prescribed surge amplitude is high.

\item Choose the order of magnitude of  $\varepsilon$ to fit the average pulse frequency.
Assuming that $a_{1}$ is small, we may approximate the minimum period of the GnRH Secreting System, for $X$ ranging between $-2\mu$ and $-\mu$, by:
\begin{equation}
T_{min}=2\varepsilon \int_{-2\lambda }^{-\lambda }\frac{f^{\prime }(x)}{a_{0}x}dx
= \frac{\varepsilon}{a_{0}} \left( 9 \lambda ^{2}-6 \ln (2) \right)
\end{equation}
Hence, to obtain a prescribed frequency $\phi $ at the end of the pulsatility phase, we can link $a_{0}$ and $\varepsilon$ by:
\begin{equation}
T_{min}=1/\phi
\label{Freq}
\end{equation}
 \item Tune the value of $a_{0}$, subject to condition \eref{Freq}, to fit the surge amplitude
to obtain the suitable surge amplitude.
\item Find the value of $c$ consistent with the pulse frequency ratio between the beginning and the end of the pulsatile phase.
With $a_{2}$ ranging between $\lambda $ and $\sqrt(3) \lambda $, the period of the limit cycle of the GnRH Secreting System for $X=-\mu$ can be approximated by the minimum $T_{min}$, whatever the value of $c$ subject to conditions \eref{NonRightHopf} and \eref{Constraintc}.  Then, the period of the limit cycle for $X=-2\mu$ is equal to $T_{min}/\rho$, where $\rho$ denotes the frequency ratio. 
Finding $c$ thus amounts to solve the implicit equation:
\[
\int_{-2\lambda }^{-\lambda }\frac{f^{\prime }(x)}
{a_{0}x+a_{1}f(x)+a_{2}-c\mu}dx+\int_{2\lambda }^{\lambda }\frac{f^{\prime }(x)}
{a_{0}x+a_{1}f(x)+a_{2}-c\mu}dx = \frac{T_{min}}{\rho}
\]
It is worth noticing that this equation does not admit a solution in $c$ for every value of $\rho$. The greatest ratio can be reached with the following value of $c$:
\begin{equation}
c = \frac{a_{0}\lambda +2a_{1}\lambda ^{3}-a_{2}}{\mu }
\label{SurHopf}
\end{equation}

\item Define the precise value of $a_1$. We have already assumed that $a_1$ is small enough (in a sense detailed in \S\ref{ordertuning}). The precise choice of $a_1$ affects marginally the amplitude of the surge, which increases as $a_{1}$ increases

\item Deduce the values of $b_1$ and $b_2$ from the results established in \S \ref{FolRS}.
For a prescribed duration ratio, there is a 1-dimensionnal curve of solutions in the $(b_{1},b_{2})$-space. Along one such curve of constant ratio,
the smaller $b_{1}$ is, the longer $X$ remains close to $-\mu$, in comparison with the time spent near $X=-2\mu$, hence the sooner the pulse frequency increases.
This property is used to set approximately the durations of the luteal and follicular phases, even if it is not  straightforward to determine the value of $X$ (between $-2\mu$ and $-\mu$) for which the pulse frequency starts increasing drastically. 
\end{enumerate}

\section{Quantitative neuroendocrinological specifications}  \label{Appli}
\subsection{GnRH secretion along the ovarian cycle in the ovine species}

\subsubsection{Duration of ovarian cycle phases}

During the breeding season (late summer to the start of the next spring), the ewe shows ovarian (estrous) cycles of 16-17 days \cite{Gordon-book}. The cycle is divided into 2 phases, the follicular and the luteal phase. The follicular phase has a duration of 2-3 days and is characterized by increasing secretion of estradiol and onset of the LH (Luteinizing Hormone) surge (triggered by the GnRH surge). The luteal phase has a duration of 14-15 days and is characterized by the secretion of progesterone from the corpus luteum. The transition from the follicular to the luteal phase is marked by ovulation.  In the absence of pregnancy, the transition from the luteal to the follicular phase is marked by the luteolysis of the corpus luteum, allowing resumption of a new cycle.

\subsubsection{Amplitude and frequency ratios}
We first detail the specifications concerning the ovine species, since this is the species for which there are the more data available directly on
the GnRH level (rather than the LH level). Indeed, in the ewe, the development of a dedicated technique has allowed the sampling of pituitary portal
blood with high temporal resolution \cite{Moenter_90}. This technique is specially useful in studying the pattern of GnRH secretion during the surge.
It has been utilized both during the luteal and follicular phase of naturally cycling animals \cite{Moenter_91} and during artificial, estradiol-induced,
follicular phases in ovariectomized ewes \cite{Moenter_91,Moenter_92}. In any case, the LH surge induced by estradiol was invariably accompanied by a
massive GnRH surge. Depending on studies and between-animal variability, the peak values of the GnRH surge\footnote{In the cited papers, the surge
peak value is actually an average computed from the values observed during one hour after the maximum has occured. Hence, with a 10 min sampling interval,
this is an average over 6 values.} ranged from 85 to 170 (73 to 394 outside the breeding season) \cite{Moenter_90}, 9 to 109 \cite{Moenter_91} fold over
baseline, with 10 min sampling intervals, or even from 100 to 500 fold over baseline with 2 min sampling intervals \cite{Moenter_92}. This variability
can be partly explained by differences in the anatomical level where the surgical cut penetrates the portal vasculature (the higher the cut  the higher
the surge amplitude). Comparatively, the average pulse peak values can be assessed as less than 10 fold the baseline. Accordingly, we chose a surge to
pulse amplitude ratio instance of $60$ in the numerical simulation illustrated below.
These studies had also established that the GnRH pulse frequency not only was greater in the early follicular phase than in the luteal phase, but also
increased further within the follicular phase as the surge approached. More precisely, during the luteal phase, the average frequency was one pulse 
per 4 hours, while it exceeded one pulse per hour in the follicular phase. Accordingly, we considered an average pulse frequency ratio of 4 between
the luteal and follicular phase in the model.

\subsubsection{Parameter combination for the ovine cycle}
Gathering the information detailed above, we set ourselves the target of meeting the specifications detailed in Table \ref{tablesheep}. Applying the procedure described in \S\ref{GnRHSecrSys}, we deduce the set of parameters listed in Table \ref{paramsheep}. The corresponding GnRH secretion pattern is illustrated in Figure \ref{signsheep}. The zoom around the surge in Figure \ref{signsheepZS} emphasizes the difference in the pulse frequency at the end of the follicular phase (pre-surge period) compared to the beginning of the luteal phase (post-durge period).
\begin{table}[ht]
\caption{Quantitative specifications for GnRH secretion pattern along the ovarian cycle in the ewe}
\begin{center}
\begin{tabular}{l|c|c}
\hline
feature&specification&unit\\
\hline
whole cycle duration&16.5&days\\
follicular phase duration&2.5&days\\
surge duration&1&days\\
luteal phase duration&13&days\\
pulse to surge amplitude ratio&60&$-$\\
frequency  increase ratio&4&$-$\\
\hline
\end{tabular}
\end{center}
\label{tablesheep}
\end{table}
\begin{table}[ht]
\caption{Parameter values for the ovarian cycle in the ewe}
\begin{center}
\begin{tabular}{c|c|c|c}
\hline
$\varepsilon$ & $0.02$  & $\delta$ & $0.0125$ \\
  $a_{0}$     & $0.52$  & $a_{1}$  & $0.011$ \\
  $a_{2}$     & $1.14$  & $c$      & $0.70$ \\
  $b_{1}$     & $0.246$ & $b_{2}$  & $1.5103$ \\
\hline
\end{tabular}
\end{center}
\label{paramsheep}
\end{table}
\begin{figure}[H]
\centering
\includegraphics[height=6.8cm]{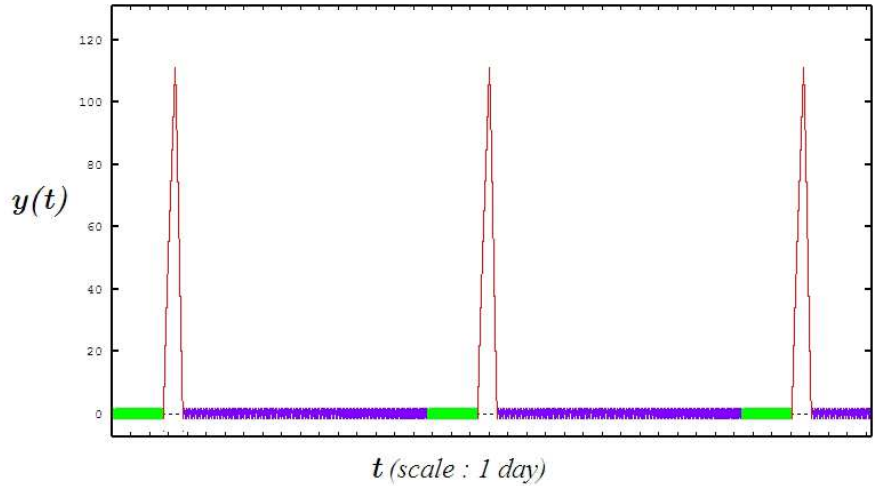}
\caption{Ovarian cycle in the ewe. $y(t)$ signal generated by \eref{Sys1} with the parameter values given in Table \ref{paramsheep}.
The time scale unit is one day. The $y$ variable has no dimension as we are interested in the pulse to surge amplitude ratio.
The signal respects the specifications for the sheep ovarian cycle :
whole cycle duration of 16.5 days, follicular phase of 2.5 days, surge duration of 1 day. The pulse to surge amplitude ratio
is around 60.}
\label{signsheep}
\end{figure}
\begin{figure}[H]
\centering
\includegraphics[height=4cm]{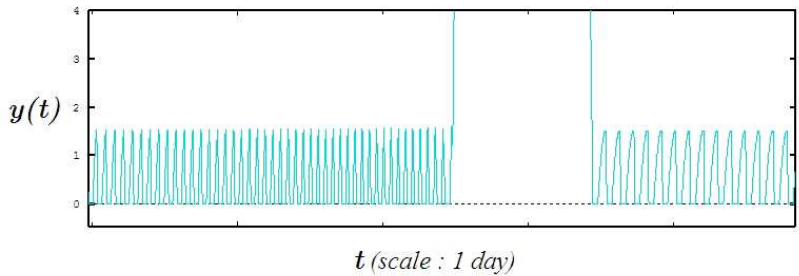}
\caption{Zoom around the surge (ewe). $y(t)$ signal generated by \eref{Sys1} with the parameter values given in Table \ref{paramsheep}.
The whole follicular phase (2.5 days) is represented here. The pulse frequency before the surge is approximatively 1 per 70 minutes.
The pulse frequency at the beginning of the luteal phase is around 1 per 2.5 hours.}
\label{signsheepZS}
\end{figure}

\subsection{GnRH secretion along the ovarian cycle in the rhesus monkey}

\subsubsection{Specifications on cycle phase duration and GnRH surge to pulse ratios}
The duration of the ovarian (menstrual) cycle in the rhesus monkey is comparable to that of the human cycle, i.e around 28 days long. Comparatively to the ovine cycle, it is much more symmetric, since the average duration of both the follicular and the luteal phase amounts to 14 days \cite{monkey-book}.

GnRH measurements in cerebrospinal fluid (CSF) samples have been obtained from the third ventricle of intact and ovariectomized conscious rhesus monkeys during control periods and throughout an estrogen challenge \cite{LXia_92}, with a time resolution of 15 min. In the intact as well as ovariectomized rhesus monkey, a genuine GnRH surge does occur in response to estradiol and the profile of the GnRH surge is remarkably similar to that reported in the ewe. These semi-quantitative information are quite reliable since simultaneous measurements of GnRH in CSF and portal blood in the ewe have shown that there is a good agreement between both techniques at the time of the GnRH surge \cite{Skinner_97}. However, it is more difficult to extract accurate quantitative information from these data, since the technique of collecting GnRH in the CSF is less reliable than the sampling of pituitary portal blood, due to difficulties in maintaining the required CSF flow uninterruptedly for a long time and documenting the position of the tip of the collecting cannula. Moreover, the whole surge duration was not entirely covered by the collecting period in several monkeys, so that maximal values may have been missed. Considering the maximal observed GnRH CSF concentration (around 400 pg/ml) and a median pulse amplitude of 16 pg/ml, we can however fix the pulse to surge amplitude ratio for instance to 25.

\begin{table}[H]
\caption{Quantitative specifications for GnRH secretion pattern along the ovarian cycle in the rhesus monkey}
\begin{center}
\begin{tabular}{l|c|c}
\hline
feature&specification&unit\\
\hline
whole cycle duration&28&days\\
follicular phase duration&13&days\\
surge duration&1&days\\
luteal phase duration&14&days\\
pulse to surge amplitude ratio&25&$-$\\
frequency  increase ratio&4&$-$\\
\hline
\end{tabular}
\end{center}
\label{tablerhesus}
\end{table}

\subsubsection{Parameter combination for the rhesus monkey cycle}
Gathering the information detailed above, we set ourselves the target of meeting the specifications detailed in Table \ref{tablerhesus}. Applying the procedure described in \S\ref{GnRHSecrSys}, we deduce the set of parameters listed in Table \ref{paramrhesus}. The corresponding GnRH secretion pattern is illustrated in Figure \ref{signrhesus}. One can notice the difference in the length of the follicular phase compared to that of the ovarian cycle in the ewe. The zoom around the surge in Figure \ref{signrhesusZS} emphasizes the difference in the pulse frequency at the end of the follicular phase (pre-surge period) compared to the beginning of the luteal phase (post-durge period).

\begin{table}[ht]
\caption{Parameter values for the ovarian cycle in the rhesus monkey}
\begin{center}
\begin{tabular}{c|c|c|c}
\hline
$\varepsilon$ & $0.018$  & $\delta$ & $0.0125$ \\
  $a_{0}$     & $0.7$    & $a_{1}$  & $0.013$ \\
  $a_{2}$     & $1.0$    & $c$      & $0.67$ \\
  $b_{1}$     & $0.187$  & $b_{2}$  & $1.704$ \\
\hline
\end{tabular}
\end{center}
\label{paramrhesus}
\end{table}

\begin{figure}[h]
\centering
\includegraphics[width=12.2cm]{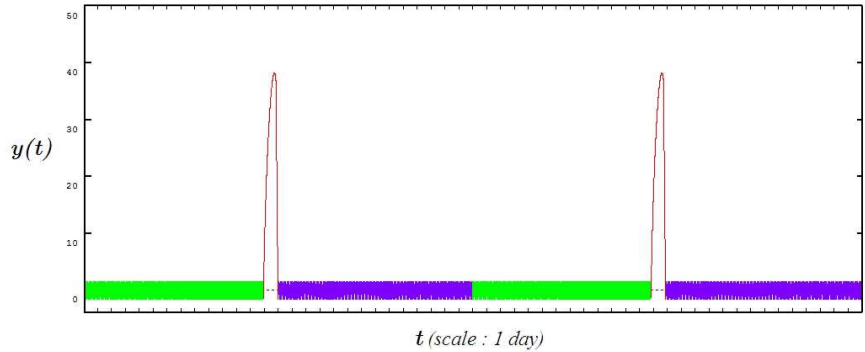}
\caption{Ovarian cycle in the rhesus monkey. $y(t)$ signal generated by \eref{Sys1} with the parameter values given in Table \ref{paramrhesus}.
The time scale unit is one day. The $y$ variable has no dimension as we are interested in the pulse to surge amplitude ratio.
The signal respects the specifications for the rhesus monkey ovarian cycle :
whole cycle duration of 28 days, follicular phase of 13 days, surge duration of 1 day. The pulse to surge amplitude ratio
is around 25.}
\label{signrhesus}
\end{figure}

\begin{figure}[H]
\centering
\includegraphics[width=12cm]{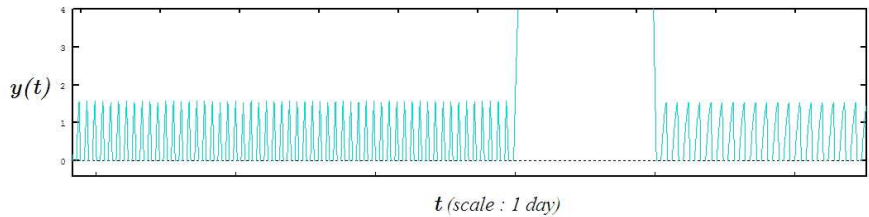}
\caption{Zoom around the surge (rhesus monkey).$y(t)$ signal generated by \eref{Sys1} with the parameter values given in Table \ref{paramrhesus}.
The three last days of the follicular phase are represented here. The pulse frequency before the surge is approximatively 1 per 80 minutes.
The pulse frequency at the beginning of the luteal phase is around 1 per 2 hours.}
\label{signrhesusZS}
\end{figure}

\section{Conclusion and discussion}

We have based our study on a concise model reproducing the different GnRH secretory patterns along an ovarian cycle \cite{fc-jpf_07}.
In this model, the dynamical pattern of GnRH secretion results from the interaction between two neuronal populations: GnRH secreting
neurons and regulating neurons, whose paragon can be embodied by Kiss-peptin neurons \cite{dungan_06}. The latter population integrates
much of the ovarian steroid feedback and acts as a slow pacemaker, while the former one is part
time excitable, part time fast oscillating. The alternation of pulsatile and surge mode for GnRH secretion, as well as pulse frequency increase
as the surge onset gets closer, have been explained in the framework of bifurcation theory.

In this paper, we have conducted a deeper analysis of this model, with the definite aim of constraining the model outputs (and more precisely
the output variable corresponding to GnRH secretion), with respect to a physiologically relevant list of specifications. In other words, we have
challenged the ability of the model to meet precise quantitative relations on the secretion signal features. Apart from the total duration of the
ovarian cycle, which is expressed in physical time, these relations can all be expressed as ratios, regarding (i) surge duration over the whole cycle duration, (ii) the duration of the luteal phase
over that of the follicular phase, (iii) pulse amplitude over surge amplitude, and (iv) pulse
frequency in luteal phase compared to follicular phase.

We have described in great details the sequence of bifurcations undergone by the uncoupled FitzHugh-Nagumo subsystems, beyond the usually-investigated
situation where the slope of the slow variable nullcline is steep. Using singular perturbation theory and dynamical analysis, we have formulated a $\varepsilon$-expansion of the homoclinic bifurcation surface in the $(b_1,b_2,\varepsilon)$ space. From this expansion, we have been able to restrict the space of parameter values search precisely, without reducing the set of reachable ratios. Within this restricted space, we have described a foliation, whose leaves define constant duration ratios between the surge and whole cycle. From then on, we have used this foliation to determine a sufficient condition to fulfill the specification on such a ratio. This condition links together the values of 3 parameters $(b_1,b_2,\varepsilon)$ over the 7 to be fixed in the adimensionned form of the model, that we have rewritten in the beginning of the article. The remaining parameters can be further
tuned according to an algorithm taking into account the other prescribed ratios. We have managed to cope with all the ratios, even if the frequency
increase is structurally limited in the framework of the model.

We have finally applied our results to reproduce the GnRH secretion pattern in two different species in which GnRH data are both directly available
and reliable. The ovine species exemplifies an estrous cycle, whose most obvious sign consists of heat behavior, while the rhesus monkey exemplifies
a menstrual cycle, whose most obvious sign consists of menstruation. The former species is mostly interesting from an agronomic viewpoint, while the
latter is interesting from a comparative physiology viewpoint, since its ovarian cycle is the closest to the human cycle. Even in species for which
fewer GnRH data are available, we can make use of our parameter search algorithm. By instance, if we assume that the GnRH secretion dynamics in cows
is comparable to that observed in ewes, as suggested in \cite{Gazal_98}, and adapt the durations of the follicular and luteal phases to values appropriate
for the cow \cite{Gordon-book-2}, we can bend the ovine GnRH signal into a bovine-like signal. 
In the same spirit, we plan to represent various physiological and pathological situations from our deep understanding of the model behavior.

From a dynamical viewpoint, this study indirectly addresses the questions of tracking homoclinic connexions, where classical Canard cycles disappear, and Bogda\-nov-\-Takens like bifurcations of Relaxation Oscillators that occur on a parameter control manifold. Future work will also focus on determining whether the whole system admits a strictly periodic attractive orbit. It is a challenging question dealing with the synchronization of weakly coupled oscillators and delay to bifurcation analysis.

\bibliographystyle{siam}
\bibliography{GNRH}
 \end{document}